\documentclass[12pt]{amsart}
\usepackage{color}
\usepackage[all,cmtip]{xy}
\usepackage{amssymb}
\usepackage{amsmath}
\usepackage{amsthm}
\usepackage{amscd}
\usepackage{amsfonts}
\usepackage{graphicx}
\usepackage{multicol}
\usepackage{setspace}

\calclayout

\newtheorem{dummy}{anything}[section]
\newtheorem{thm}[dummy]{Theorem}

\newtheorem{lem}[dummy]{Lemma}
\newtheorem{prop}[dummy]{Proposition}
\newtheorem{cor}[dummy]{Corollary}

\theoremstyle{definition}

  \newtheorem{ex}[dummy]{Example}
  \newtheorem{rem}[dummy]{Remark}

  \newtheorem*{ack}{Acknowledgement}

\newcommand{\bA}{\mathbf A}
\newcommand{\bB}{\mathbf B}
\newcommand{\bC}{\mathbf C}
\newcommand{\bD}{\mathbf D}
\newcommand{\bE}{\mathbf E}
\newcommand{\bF}{\mathbf F}

\newcommand{\bM}{\mathbf M}

\newcommand{\bQ}{\mathbf Q}
\newcommand{\bP}{\mathbf P}

\newcommand{\bk}{\mathbf k}

\newcommand{\bbF}{\mathbb F}

\newcommand{\bbZ}{\mathbb Z}

\DeclareMathOperator{\Hom}{Hom} 
\DeclareMathOperator{\bHom}{{\bf Hom}} 
\DeclareMathOperator{\PHom}{PHom}

\DeclareMathOperator{\Ext}{Ext}

\DeclareMathOperator{\Ann}{Ann}

\newcommand{\G}{\Gamma}
\newcommand{\Sig}{\Sigma}
\newcommand{\z}{\zeta}
\newcommand{\e}{\epsilon}
\newcommand{\Lz}{L_{\z}}
\newcommand{\Uz}{U_{\z}}
\newcommand{\bPz}{\bP (\z)}

\newcommand{\Om}{\Omega}
\newcommand{\D}{\Delta}
\newcommand{\al}{\alpha}
\newcommand{\ot}{\otimes}

\newcommand{\la}{\langle}
\newcommand{\ra}{\rangle}

\newcommand{\bd}{\partial}
\newcommand{\id}{\mathrm{id}}

\newcommand{\proj}{\mathrm{proj}}

\def\maprt#1{\smash{\,\mathop{\longrightarrow}\limits^{#1}\,}}

\begin{document}

\title{Productive Elements in Group Cohomology}

\author{Erg{\" u}n Yal\c c\i n }

\address{Department of Mathematics, Bilkent
University, Ankara, 06800, Turkey.}

\email{yalcine@fen.bilkent.edu.tr}

\keywords{Group cohomology, chain complex, diagonal approximation}

\thanks{2000 {\it Mathematics Subject Classification.} Primary: 20J06; Secondary: 20C20, 57S17. \\ The author is partially supported by T\" UB\. ITAK-TBAG/110T712}

\date{\today}

\begin{abstract}
Let $G$ be a finite group and $k$ be a field of characteristic $p>0$. A cohomology class $\z \in H^n(G,k)$ is called productive if it annihilates $\Ext^*_{kG}(\Lz ,\Lz)$. We consider the chain complex $\bPz$ of projective $kG$-modules which has the  homology of an $(n-1)$-sphere and whose $k$-invariant is $\z$ under a certain polarization. We show that $\z$ is productive if and only if there is a chain map $\Delta: \bPz \to \bPz \otimes \bPz$ such that $(\id \otimes \epsilon)\Delta\simeq \id$ and $(\epsilon \otimes \id)\Delta \simeq \id$. Using the Postnikov decomposition of $\bPz \ot \bPz$, we prove that there is a unique obstruction for constructing a chain map $\D$ satisfying these properties. Studying this
obstruction more closely, we obtain theorems of Carlson \cite{Paper:Carlson1987Products} and Langer \cite{Paper:Langer2010DyerLashoff} on productive elements.
\end{abstract}

\maketitle

\section{introduction}
\label{sect:Introduction}

Let $G$ be a finite group and $k$ be a field of characteristic $p>0$. Let $\z \in H^n(G,k)$ denote a nonzero cohomology class of degree $n$ where $n\geq 1$. Associated to $\z$, there is a unique $kG$-module homomorphism $\hat \z : \Omega ^n k\to k$ and the $kG$-module $\Lz$ is defined as the kernel of this homomorphism. A cohomology class $\z$ is called productive if it annihilates the cohomology ring $\Ext^*_{kG}(\Lz,\Lz)$. In this paper, we study the conditions for a cohomology class to be productive.

Under the usual identification of $H^n(G,k)$ with the group $\mathcal{U}^n (k,k)$ of $n$-fold $kG$-module extensions of $k$ by $k$, the cohomology class $\z$ is the extension class of an extension of the form
\begin{equation}\label{eqn:extension}
0 \to k \to P_{n-1}/ \Lz \to P_{n-2} \to \cdots \to P_0 \to k \to 0
\end{equation}
where $P_0, \dots, P_{n-1}$ are projective $kG$-modules. Let $\bC_{\z}$ denote the chain complex obtained by truncating both ends of this extension. Splicing $\bC _{\z}$ with itself infinitely many times (in the positive direction), one
obtains a periodic (positive) chain complex $\bC _{\z} ^{\infty}$ and the tensor product of these complexes $\otimes _i \bC ^{\infty} _{\z_i}$ over a set of cohomology classes $\{ \z_1 , \dots, \z_r \}$ is called a multiple complex.
It is shown in \cite{Paper:BensonCarlson87Complexity} that a multiple complex gives a projective resolution of $k$ as a $kG$-module if and only if $\{\z_1,\dots,\z_r \}$ is a system of parameters for the cohomology ring $H^*(G, k)$.

In \cite{Paper:Carlson1987Products}, Carlson studies a complex dual to $\bC_{\z}$. Let $\bD_{\z}$ denote the chain complex which is obtained by first taking the dual of $\bC_{\z}$ and then shifting it to the left so that $(P_{n-1}/ \Lz )^*$ is at
dimension zero. In a similar way,  we can form an infinite complex $\bD_{\z}^{\infty}$ by splicing $\bD _{\z}$ with itself infinitely many times in the positive direction. Note that the complex $\bD _{\z} ^{\infty} $ has an
augmentation map $\e : \bD _{\z} ^{\infty} \to k$ which comes from the dual of the map on the left side of the extension (\ref{eqn:extension}). Carlson proves the following:

\begin{thm}[Carlson \cite{Paper:Carlson1987Products}]\label{thm:Carlsonmainthm} Let $\z \in H^n (G,k)$ be a nonzero cohomology class of degree $n$ where $n \geq1$. Then, there is a chain map $\phi : \bD _{\z} ^{\infty}\to \bD _{\z} ^{\infty}\ot \bD _{\z} ^{\infty}$ which satisfies $(\id\ot \e )\phi=\id$ and $(\e \ot \id)\phi=\id$ if and only if $\z$ is productive.
\end{thm}

Note that if there is a chain map $\phi$ as in the above theorem, then the map induced by $\phi$ on cohomology defines a product with unity on the cohomology of $\bD_{\z}^{\infty}$. This was the main motivation for Carlson to study the productive elements since, when there is a product structure on $H^*(\bD_{\z}^{\infty})$, it is easier to calculate the differentials in the hypercohomology spectral sequence for multiple complexes.

In this paper, we consider the chain complex $\bPz$ of projective $kG$-modules whose homology is the same as the homology of an $(n-1)$-sphere and whose $k$-invariant
is $\zeta$ under a certain polarization. Alternatively, one can define $\bP (\z)$ as follows: Let $\bP$ be a projective resolution of $k$ as a $kG$-module and $\Sigma ^{n-1} \bP$ denote the chain complex where $(\Sig ^{n-1}\bP)_i=\bP _{i-n+1}$ and $\bd=(-1)^{n-1} \bd $. Then, $\bPz$ is defined as the chain complex that fits into an extension of the form
\begin{equation}
0 \to \Sigma ^{n-1} \bP \to  \bP (\z) \to \bP \to 0
\end{equation}
whose extension class is $\z$ under the identification $[\bP, \Sig ^n \bP]=\Ext ^n _{kG} (k,k)$ (see Section \ref{sect:Extensions} for more details). Note that the complex $\bP (\z)$ has an augmentation map $\epsilon : \bP (\z)\to k$ induced from the augmentation map of $\bP$. Our first result is the following:

\begin{thm}\label{thm:main1} Let $\z \in H^n (G,k)$ be
a nonzero cohomology class of degree $n$ where $n \geq 1$ and let $\bP(\z)$ denote the chain complex of projective $kG$-modules which has the homology of an $(n-1)$-sphere and whose $k$-invariant is $\z$. Then, $\z$ is productive if and only if there is a chain map $\Delta: \bPz \to \bPz \otimes \bPz$ which satisfies $(\id \otimes \epsilon)\Delta\simeq \id$ and $(\epsilon \otimes \id)\Delta \simeq \id$.
\end{thm}

The proof essentially follows from the observation that $\bP (\z)$ is a projective resolution of the chain complex $\bC_{\z}$. Since there is a chain map $\bC_{\z} \to \bD _{\z}$ inducing an isomorphism on homology (see \cite[Proposition 5.2]{Paper:BensonCarlson94ProjectiveResolutions}), this implies that $\bP (\z)$ is also a projective resolution for $\bD _{\z}$. Then Theorem \ref{thm:main1} follows from Theorem \ref{thm:Carlsonmainthm} as a consequence of some standard results on projective resolutions.

Next, we consider the question when there is a chain map $\Delta : \bP (\z) \to \bP (\z) \otimes \bP (\z)$ satisfying $(\id \otimes \epsilon)\Delta \simeq \id$ and $(\epsilon \otimes \id)\Delta \simeq \id $. We answer this question by considering the Postnikov decomposition of $\bP (\z) \otimes \bP (\z)$ (see Dold \cite{Paper:Dold60ZurHomotopie}). We observe that $\bPz \ot \bPz$ fits into an extension of the form
\begin{equation}\label{eqn:postnikov1}
0 \to \Sig ^{n-1} \bP \otimes \Sig ^{n-1} \bP \to \bP (\z)
\otimes \bP (\z) \maprt{\pi} \bP (\z \oplus \z) \to 0
\end{equation}
where $\bP(\z \oplus \z)$ is defined as the chain complex that fits into an extension
\begin{equation}\label{eqn:postnikov2}
0 \to (\Sig ^{n-1} \bP \otimes \bP) \oplus (\bP \otimes \Sig^{n-1} \bP) \to \bP (\z \oplus \z ) \to \bP \otimes \bP \to 0
\end{equation}
with extension class $\theta =( \z \times \id , \id \times \z )$. It turns out that one can always find a chain map $\psi: \bP (\z) \to \bP (\z \oplus \z)$
that commutes with the diagonal approximation for $\bP$ (see Proposition \ref{prop:firstlifting} for the definition of $\psi$). The existence of a chain map $\D : \bPz \to \bPz \ot \bPz$ satisfying the properties in Theorem \ref{thm:main1} is equivalent to the existence of a lifting $\tilde \psi$ of $\psi$ satisfying $\pi \tilde \psi=\psi$ where $\pi$ is the surjective map given in (\ref{eqn:postnikov1}). There is a unique obstruction for such a lifting and after studying this obstruction, we prove the following:

\begin{thm}\label{thm:main2} Let $k$ be a field of characteristic $p>0$ and let $\z \in H^n (G,k)$ be a nonzero cohomology class of degree $n$ where $n\geq 1$.
If\ $p>2$ and $n$ is even, then $\psi: \bP (\z) \to \bP
(\z\oplus \z)$ lifts to a chain map
$\tilde \psi : \bPz \to \bPz \ot \bPz$ satisfying $\pi \tilde \psi=\psi$. For $p=2$, this lifting exists if and only if $\widetilde{Sq} ^{n-1} \z $ is a multiple of $\z$.
\end{thm}

Here $\widetilde{Sq} ^{n-1}$ denotes the semilinear extension of the Steenrod square $Sq^{n-1}$. Note that $H^*(G,k)\cong k \ot _{\bF _2} H^*(G,\bF _2 )$ is a $k$-vector space with basis $\{u_i\}$ lying in $H^*(G,\bF_2)$. The semilinear extension of $Sq^{n-1}$ action on $H^*(G,k)$ is defined by $$\widetilde{Sq}^{n-1} (\sum _i
\lambda _i u_i )= \sum _i \lambda_i ^2 Sq ^{n-1} u_i.$$ The reason for taking the semilinear extension instead of the usual Steenrod square action is explained in detail in Section \ref{thm:main2}. We also give an example at the end of Section \ref{thm:main2} to illustrate the importance of this point (see Example \ref{ex:semilinear extension}).

As a corollary of Theorem \ref{thm:main2}, we obtain theorems of Carlson \cite[Theorem
4.1]{Paper:Carlson1987Products} and Langer \cite[Theorem
6.2]{Paper:Langer2010DyerLashoff} on productive elements.

\begin{cor}[Carlson \cite{Paper:Carlson1987Products}, Langer
\cite{Paper:Langer2010DyerLashoff}] Let $k$ be a field of
characteristic $p>0$ and let $\z \in H^n (G,k)$ be a nonzero
cohomology class of degree $n$ where $n \geq 1$. Then, the following holds:
\begin{enumerate} \item If $p>2$ and $n$ is even, then $\z$ is productive.
\item If $p=2$, then $\z$ is productive if and only if $\widetilde{Sq} ^{n-1} \zeta$
is a multiple of $\zeta$.
\end{enumerate}
\end{cor}

In the rest of the paper, we  consider the question when a cohomology class $\z \in H^n(G,k)$ annihilates $\Ext ^* _{kG} (L_{\zeta}, k)$. We observe that
this is actually a weaker condition than being productive (see Example \ref{ex:semi but not prod}), so we call such a cohomology class {\it semi-productive}. We relate being semi-productive to Massey products and then to Steenrod squares. We prove the following:

\begin{thm}\label{thm:main3} Let $k$ be a field of characteristic $2$ and let $\z \in H^n (G,k)$
be a nonzero cohomology class of degree $n$ where $n\geq 1$. Then, the following are equivalent: \\ (i) $\z$ is semi-productive. \\ (ii) For every $v\in H^*(G,k)$ satisfying  $v \zeta=0$, the Massey product $\la \z,v, \z \ra \equiv 0$ mod $(\z)$. \\
(iii) For every $v\in H^*(G,k)$ satisfying $v\zeta=0$, the product $v
\widetilde{Sq}^{n-1} \zeta \equiv 0$ mod $(\z)$.
\end{thm}

The equivalence of the last two statements follows from a theorem of Hirsch \cite{Paper:Hirsch55Quelques} which says that if $X$ is a simplicial complex, then for every $v,\z \in H^* (X, \bF _2)$, the equation $$\la \z , v, \z \ra \equiv v Sq^{n-1} \z \ \  \mod (\z)$$ holds. At the end of the paper, we give an example of a cohomology class which
is not semi-productive. We also provide an example of a semi-productive element which is not productive.

Throughout the paper all modules are finitely-generated. Whenever there is more than one place to find a theorem and its proof, we refer to the original paper although it may not be the easiest source to find. Many of the old results that we quote in the paper can also be found in the books by Benson \cite{Book:Benson95Rep&CohI}, \cite{Book:Benson95Rep&CohII} and Carlson \cite{Book:Carlson1996Modules}, \cite{Book:carlson03CohofGroups}.

\ack{I thank Martin Langer for many helpful conversations on the topic and for directing me to Hirsch's work. I also thank him for providing me an example
of a cohomology class which is not semi-productive (Example \ref{ex:MLangerexample}).} I also thank the referee for his/her corrections and suggestions on the paper.

\section{Preliminaries on chain complexes}
\label{sect:preliminaries}

In this section we introduce our notation for chain complexes and state some well-known results about hypercohomology of chain complexes. For more details on this material, we refer the reader to 
\cite[Sections 2.3 and 2.7]{Book:Benson95Rep&CohI} and
\cite[Section 2]{Paper:BensonCarlson94ProjectiveResolutions}.

Let $G$ be a finite group and $k$ be a field of characteristic $p>0$. Throughout the paper, whenever we say $\bC$ is a chain complex, we always mean that $\bC$ is a chain complex of finitely-generated $kG$-modules and it is bounded from below, i.e., there exists an $N$ such that $\bC _i=0$ for $i<N$. In fact, almost all of our chain complexes are positive, i.e., $\bC_i=0$ for $i<0$. 

Let $\bC$ and $\bD$ be two chain complexes. Then, we denote by $\bC \otimes \bD$ the chain complex $$(\bC \otimes \bD)_n= \bigoplus_{i+j=n} \bC_i \otimes _k \bD _j$$ whose differential is defined by $\bd(x\otimes y)=\bd (x) \otimes y + (-1) ^i x \otimes \bd (y)$ for every $x \in\bC_i$ and $y\in \bD_j$. The $G$-action on $\bC \otimes \bD$ is given by the diagonal action. We define the Hom-complex $\bHom (\bC, \bD)$ as the chain complex
$$\bHom (\bC, \bD)_n =\prod _{i+n=j} \Hom _{kG} (\bC _i, \bD _j)$$ with differential $\bd(f) (x)= \bd (f(x))-(-1) ^n f (\bd (x))$ for $f \in \Hom _{kG} (\bC _i, \bD _{i+n} )$ and $x\in \bC _{i+1}$. If $f \in \bHom (\bC, \bD)_n $, then we say $f$ is of degree $n$. A map $f:\bC \to \bD$ of degree zero is called a chain map if $\bd(f)=\bd f - f\bd=0$. We say two chain maps $f,g : \bC \to \bD$ are homotopic if there is a degree one map $H: \bC \to \bD$ such that $\bd H +H \bd=f-g$. In this case, we write $f \simeq g$. We denote by $[\bC, \bD]$ the group of homotopy classes of chain maps $f: \bC \to \bD$. Note that $[\bC, \bD]$ is the same as $0$-th homology group of the
Hom-complex $\bHom (\bC, \bD)$.

A chain map $f:\bC \to \bD$ is called a homotopy equivalence if there exists a chain map $g: \bD \to \bC$ such that $fg\simeq \id$ and $gf\simeq \id$. If a chain map $f$ induces isomorphism on homology, then we say $f$ is a weak (homology) equivalence. For each integer $n$, we denote by $\Sigma ^n \bC$, the chain complex $(\Sigma ^n \bC)_i=\bC _{i-n}$ with differential $(\Sigma^n \bd) _i=(-1)^n \bd_{i-n}$. Note that $\bHom (\bC, \Sig ^n \bD)=\Sig ^n \bHom (\bC,\bD)$. So we have
\begin{equation}\label{eqn:identifications}
[\bC, \Sig ^n \bD]=H_0 (\Sig ^n \bHom (\bC,\bD))=H_{-n} (\bHom(\bC, \bD))=H^n (\bHom (\bC, \bD ))
\end{equation}
where the last equality comes from the usual convention of interpreting a chain complex as a cochain complex by taking $C^n=C_{-n}$ and $\delta ^n =\bd _{-n}$.

\subsection{Hypercohomology}
Given a chain complex $\bC$, a projective resolution of $\bC$ is defined as a chain complex $\bP$ (bounded from below) of projective $kG$-modules together with a chain map $\bP \to \bC$ which induces an isomorphism on homology. Given two chain complexes $\bC$ and $\bD$, the $n$-th ext-group of $\bC$ and $\bD$ is defined as $$ \Ext ^n _{kG} (\bC, \bD):=H^n (\bHom (\bP, \bD))$$ where $\bP \to \bC$ is a projective resolution of $\bC$. The ext-group $\Ext ^n_{kG} (\bC, \bD)$ is called the $n$-th hypercohomology group of $\bC$ and $\bD$. Using the identification given in $(\ref{eqn:identifications})$, we can also take
\begin{equation*}
\Ext ^n _{kG} (\bC, \bD)=[\bP, \Sig ^n \bD]
\end{equation*}
where $\bP$ is a projective resolution of $\bC$. The following is a useful observation:

\begin{lem}\label{lem:usefulobservation}
If $f: \bC \to \bC'$ and $g:\bD \to \bD'$ are weak equivalences, then the induced map
$$(f^*, g_*) : \Ext ^* _{kG} (\bC ', \bD) \to \Ext ^* _{kG} (\bC, \bD ')$$ is an isomorphism.
\end{lem}

\begin{proof} Let $\bP$ and $\bP'$ be projective resolutions of $\bC$ and $\bC '$, respectively. Then, $f$ lifts to a chain map $\tilde f: \bP \to \bP '$ which is also a weak equivalence. Since $\bP$ and $\bP'$ are projective complexes which are bounded from below, 
$\tilde f$ is a homotopy equivalence (see 
\cite[Chp I, Thm 8.4]{Book:Brown1982CohGrps}). This 
induces a homotopy equivalence 
$$f^*:  \bHom (\bP ' , \Sig ^n \bD) \maprt{\simeq} \bHom (\bP, \Sig ^n \bD )$$ and hence an isomorphism on homology. Since $\Sig ^n g : \Sig ^n \bD \to \Sig ^n \bD'$ is a weak equivalence, the induced map
$$ g_*: \bHom ( \bP, \Sig ^n \bD ) \maprt{} \bHom (\bP, \Sig ^n \bD' )$$ is also a weak equivalence (see \cite[Chp I, Thm 8.5]{Book:Brown1982CohGrps}). Combining these, we get the desired isomorphism on ext-groups.
\end{proof}

When $\bC$ is a chain complex of projective $kG$-modules and $M$ is a $kG$-module, the $n$-th cohomology of the cochain complex $\Hom _{kG} (\bC, M)$ is often denoted by $H^*(\bC, M)$. Note that the cohomology group $H^n (\bC, M)$ is the same as the hypercohomology group $\Ext^n _{kG}(\bC, \bM)$, where $\bM$ is the chain complex with $M$ at dimension zero and zero everywhere else. So, using Lemma
\ref{lem:usefulobservation}, we can identify the cohomology group $H^*(\bC, M)$ with the group of homotopy classes $[\bC, \Sig ^n \bP(M)]$. In particular, the group cohomology $H^*(G,k)$ can be identified with the group $[\bP, \Sig ^n \bP]$. More generally, for any $kG$-modules $N$ and $M$, we can identify the ext-group $\Ext ^n _{kG} (N,
M)$ with the group of homotopy classes $[\bP (N), \Sigma ^n \bP (M)]$ where $\bP (N)$ and $\bP (M)$ are projective resolutions of $N$ and $M$, respectively.

\subsection{Products in cohomology}
The algebra structure of $H^*(G,k)$ and the $H^*(G,k)$-module
structure of $H^* (\bC, k)$ can be defined in terms of composition of chain maps using the above identifications. Given $x \in H^n (G, k)$ and $y\in H^m(G,k)$, let $\hat x: \bP\to \Sig ^n \bP$ and $\hat y: \bP \to \Sig ^m \bP$ be chain maps that represent $x$ and $y$, respectively. Then, the cup product $xy \in H^{n+m} (G,k) $ is defined as the cohomology class represented by the composition
\begin{equation*}
\begin{CD} \bP @>{\hat y}>> \Sig ^m \bP @>{\Sig^m \hat
x}>> \Sig^{m+n} \bP \\
\end{CD}
\end{equation*}
Similarly, for $x\in H^n(G, k)$ and $u \in H^m (\bC , k)$, one can define the cup product $xu$ as a composition of associated chain maps.

Alternatively, one can define products in hypercohomology using (cross) products of chain maps. Given two maps $f\in \bHom( \bC, \bD)$ and $g\in \bHom (\bE ,\bF)$, the product $f\times g \in \bHom(\bC \otimes \bE , \bD \otimes \bF)$ is defined by $$(f \times g)(x \otimes y)=(-1)^{\deg (x) \deg (g)} f(x) \otimes g(y).$$ Note that here we use the following well-known sign convention: multiply an expression by $(-1)^{nm}$ whenever two terms with degrees $n$ and $m$ are swapped. In particular, we have
$$(f'\times g')\circ (f \times g)=(-1)^{\deg (g') \deg(f)} (f' \circ f)\times (g' \circ g).$$ A list of similar formulas can be found in \cite[Section 2]{Paper:Siegel97CohSplitExt}.

Using the cross product, one can express the cup product $xy$ as a composition $$ \bP \maprt{\D} \bP \otimes \bP \maprt{\hat x \times \hat y} \Sig ^n \bP \otimes \Sig^m \bP \maprt{\mu} \Sig^{m+n} \bP $$ where $\D :\bP \to \bP \otimes \bP$ is a chain map covering the diagonal map $k \to k \otimes k$ defined by $a \to a \otimes 1$ and $\mu$ is a chain map covering the multiplication map $k \otimes k \to k$ defined by $\mu (a \otimes b)=ab$. For more details on the products in hypercohomology, we refer the reader to \cite[Section 3.2]{Book:Benson95Rep&CohI}.

\section{Extensions of projective chain complexes}
\label{sect:Extensions}

Let $G$ be a finite group and $k$ be a field of characteristic
$p>0$. As in the previous section, we only consider chain complexes of finitely-generated $kG$-modules which are bounded from below. In this section, we also assume that all the chain complexes are projective, i.e., they are chain complexes of projective $kG$-modules.

Given an extension of (projective) chain complexes $0 \to \bA \to \bB \to \bC \to 0$, associated to it, there is an extension class $\alpha \in [\bC , \Sig \bA]$ defined as follows: Given an extension, we can choose $kG$-module splittings for each $n$ and assume that $\bB_n =\bA_n \oplus \bC_n$ for all $n$. Then, the differential $\bd^B$ is of the form
$$\bd ^B=\left[\begin{matrix}\bd ^A & \alpha \\ 0 & \bd ^C
\end{matrix}\right]$$ where $\alpha: \bC \to \Sig \bA$
is a chain map. The fact that $\al$ is a chain map follows from the identity $(\bd ^B )^2=0$ which gives $-\bd^A\alpha=\alpha \bd^C$. In the usual way one can define an equivalence relation for extensions and then obtain a bijective correspondence between the group of equivalence classes of extensions of the form $0 \to \bA \to \bB \to \bC \to 0$ and the group $[\bC , \Sig \bA]$, of homotopy classes of chain maps
$\bC \to \Sig \bA$. We leave the details to the reader.

One important example of an extension that we deal with in this paper is the following:

\begin{ex} Let $\zeta \in H^n (G,k)$ be a nonzero cohomology class of degree $n$ where $n\geq 1$ and let $\bP (\z)$ denote the chain complex which fits into the extension $$ 0 \to \Sig ^{n-1} \bP \to \bP (\z) \to \bP \to 0$$ whose extension class is equal to $\z$ under the identification $H^n(G,k)=[\bP, \Sig ^n \bP]$. Note that $\bP (\z)$ has the homology of an $(n-1)$-dimensional sphere. Given a chain complex $\bC$ of projective modules which has homology of an $(n-1)$-sphere, one can choose a pair of isomorphisms $ \varphi : H_0 (\bC) \to k$ and $\phi : H_{n-1} (\bC)\to k$ which is called a polarization of $\bC$, and using this polarization, one can define a unique cohomology class in $\Ext^n_{kG} (k,k)$. This cohomology class is called the $k$-invariant of the polarized complex $\bC$ (see Definition 3 in \cite{Paper:Carlsson82Inventiones}). Note that there is an obvious polarization for $\bP (\z)$ coming from the augmentation map $\e : \bP \to k$ and its shift $\Sig ^n \e : \Sig ^n \bP \to k$, and under this polarization, the $k$-invariant of $\bPz$ is equal to $\z$.
\end{ex}

We now prove some simple but useful lemmas on the extensions of projective chain complexes.

\begin{lem}\label{lem:obstruction for lifting of maps} Let $0 \to \bA \to \bB \to \bC \to 0$ be an extension with extension class $[\alpha]$ where $\alpha :\bC\to \Sig \bA$. Given a chain map $f:\bD \to \bC$, it lifts to a chain map $\tilde f : \bD \to \bB$ if and only if the composition $ \alpha f$ is homotopic to zero.
\end{lem}

\begin{proof} Choosing a $kG$-module splitting
for each $n$, we can assume that $\bB _n = \bA_n \oplus \bC_n$ for each $n$. Suppose that $f$ lifts to $\tilde f$, then we can write $\tilde f : \bD \to \bB$ as a pair $\tilde f=(H, f)$ where $H: \bD \to \bA$. The chain map condition for $\tilde f$ gives $-\bd H+H\bd=\al f$. So, $\al f: \bD \to \Sig \bA$ is homotopy equivalent to zero. Conversely, if $\al f \simeq 0$, then there is an $H$ satisfying $-\bd H +H \bd =\al f$. Taking $\tilde f=(H, f)$, we obtain a lifting for $f$.
\end{proof}

Another useful lemma is the following:

\begin{lem}\label{lem:kernel and extensions}
Let ${\mathcal E}: 0 \to \bA \maprt{i} \bB \maprt{\pi} \bC \to 0$ be an extension of chain complexes. If $$0 \to \Sig ^{-1} \bC \to \bA' \to \bB \to 0$$ is an extension with extension class $[\pi]\in [\bB, \bC]$, then $\bA$ and $\bA'$ are  homotopy equivalent.
\end{lem}

\begin{proof} We can assume that $\bB _n= \bA _n \oplus \bC_n$ for all $n$ and the differential $\bd ^B$ is of the form $$\bd ^B=\left[\begin{matrix}\bd ^A & \alpha \\
0 & \bd ^C\end{matrix}\right]$$ where $\alpha : \bC \to \Sig \bA$ is a chain map representing the extension class for $\mathcal E$. Similarly, we can take $\bA '$ as the complex where $\bA'_n= \bC _{n+1} \oplus \bA _n \oplus \bC _n$ for all $n$ and with differential $$\bd ^{A'}=\left[\begin{matrix} -\bd ^C & 0 & \id\\ 0 & \bd ^A & \alpha \\ 0 & 0 & \bd ^C \end{matrix}\right].$$ Let $j: \bA \to \bA'$ be the inclusion defined by $a \to (0, a, 0)$ and let $q: \bA ' \to \bA$ be the projection map given by $(c_1, a, c_2)\to a-\alpha (c_1)$. Taking $H: \bA' \to \Sig ^{-1} \bA'$ as the map $H(c_1, a, c_2)=(0, 0, c_1)$, we see that the equality $$\bd ^{A'} H+H \bd ^{A'} = \id - j q$$ holds. So, $j: \bA \to \bA'$ is a homotopy equivalence.
\end{proof}

The following lemma will be used in the proof of Theorem
\ref{thm:main2}.

\begin{lem}\label{lem:maps from the middle}
Let $0 \to \bA \to \bB \to \bC \to 0$ be an extension of chain complexes with extension class $[\al] \in [\bC , \Sig \bA]$. Given a chain map $\varphi : \bB \to \bD$, we can write $$\varphi =[\varphi_1 \ \varphi_2]$$ by choosing a $kG$-module decomposition for $\bB$. Then, the following holds:
\begin{enumerate}
\item $\varphi _1$ is a chain map. If $\varphi_1=0$, then $\varphi _2$ is also a chain map.
\item If $\varphi _1 \simeq 0$ via a homotopy $H$, then $\varphi \simeq [0 \ \varphi_2']$ where $\varphi_2'= \varphi_2-H\al.$
\item If $\varphi_1=0$, then $\varphi \simeq 0$ if and only if $\varphi_2 \simeq u\al$ for some chain map $u:\bA \to \Sig ^{-1} \bD$.
\end{enumerate}
\end{lem}

\begin{proof} The first statement is obvious. For (ii), take $G=[H \ 0]$, then
$G \bd +\bd G=[\varphi_1\ H\al]=\varphi-[0 \ \varphi_2']$. The last statement follows from (ii).
\end{proof}

\section{Proof of Theorem \ref{thm:main1}}
\label{sect:main1}

Let $G$ be a finite group and $k$ be a field of 
characteristics $p>0$. Given a $kG$-module $M$, there is a 
projective cover $q_M:P_M \to M$ for $M$, and the $kG$-
module $\Omega M$ is defined as the kernel of this 
surjective map. Inductively, one can define $\Omega^n M$ 
for all $n\geq 0$ by taking $\Omega ^0M=M$ and $\Omega ^n 
M= \Omega (\Omega ^{n-1} M)$ for $n \geq 1$. Note that if 
$$ \cdots \to P_n \maprt{\partial _n} P_{n-1}\maprt
{\partial_{n-1}} P_{n-2} \to \cdots \to P_1 \maprt{\partial _1} P_0 \to 
M\to 0$$ is a projective resolution of $M$, the kernel of 
$\partial_{n-1} $ is isomorphic to $\Omega ^n M \oplus Q$ 
for some projective $kG$-module $Q$. If $P_*$ is a minimal 
projective resolution, then we have  $\ker \bd _n \cong 
\Omega ^n M$ for all $n$.

Let $P_*$ be a minimal projective resolution for $k$. A 
cohomology class $\z \in H^n(G, k)$ is represented by a 
homomorphism $\hat \z :P_n \to k$ which satisfies the 
cocycle condition $\delta \hat \z=0$. So, $\hat \z$ defines 
a map $\hat \z : \Omega ^n k \to k$ and the $kG$-module $
\Lz$ is defined as the kernel of this homomorphism. It is 
easy to show that $\Lz$ is uniquely defined by $\z$ up to 
isomorphism (see \cite[Lemma 6.10]
{Book:Carlson1996Modules}). As a consequence of the
definition, we have the following diagram:
\begin{equation*}
\begin{CD}
@. L_{\zeta} @=  L_{\zeta}\\
@.  @VVV   @VVV\\
0 @>>>  \Omega^n (k) @>>> P_{n-1} @>>> P_{n-2} @>>>
\cdots @>>> k @>>> 0\\
@. @V{\hat{\zeta}}VV  @VVV  @|@|@|  \\
0 @>>> k @>>> P_{n-1}/ L_{\zeta} @>>> P_{n-2} @>>>
\cdots @>>> k @>>> 0 \ .
\end{CD}
\end{equation*}
The extension class of the extension on the bottom row of the above diagram is equal to $\z$ under the usual isomorphism between $H^n(G,k)$ and the algebra $\mathcal{U}^n (k,k)$ of $n$-fold $k$ by $k$ extensions.

The correspondence between $\z$ and the homomorphism $\hat \z : \Omega ^n k \to k$ can be made more explicit using the stable module category. Recall that the stable module category of $kG$-modules is a category where the objects are $kG$-modules and morphisms are given by $$\underline{\Hom} _{kG} (M, N)= \Hom _{kG} (M, N)/\PHom_{kG} (M,N)$$ where $\PHom _{kG} (M,N)$ denotes the subgroup of all $kG$-homomorphisms $M \to N$ that factor through a projective module. For a positive $n$, we have $$\Ext ^n _{kG} (M, N) \cong{\underline \Hom} _{kG} (\Om ^n M, N)$$ and under this identification $\z$ corresponds to the map $\hat \z : \Omega ^n k\to k$ in stable module category (see \cite[page 38-39]{Book:Carlson1996Modules}).

In \cite{Paper:Carlson1987Products}, Carlson considers the dual of the $k$ by $k$ extension given in the above diagram. This is an extension of the form
\begin{equation}\label{eqn:dualext}
0 \to k \to P_0^* \to \dots \to P_{n-2}^* \to
U_{\z} \to k \to 0
\end{equation}
where $U_{\z}=(P_{n-1}/ \Lz )^*$. The homomorphism $\Uz \to k$ is denoted by $\epsilon$. Carlson proves the following:

\begin{prop}\label{lem:splitting for Uz}
If $\zeta$ is productive, then there exists a homomorphism $\phi :\Uz \to \Uz \otimes \Uz $ such that $(\id \otimes \e ) \phi =\id=(\e \otimes \id) \phi$.
\end{prop}

\begin{proof} See \cite[Prop 2.3]{Paper:Carlson1987Products}.
\end{proof}

Let $\bD _{\z}$ denote the chain complex obtained by truncating both ends of the extension given in  (\ref{eqn:dualext}). Note that $\bD _{\z}$ has an augmentation map $\e : \bD _{\z} \to k$ which comes from $\e :\Uz \to k$. We have the following:

\begin{prop}\label{prop:slitting for Cz}
Let $\z \in H^n (G,k)$ be a nonzero cohomology class of
degree $n$ where $n\geq 1$. If $\z$ is productive, then
there exists a chain map $\phi'  : \bD _{\z} \to \bD _{\z}
\otimes \bD _{\z} $ such that $(\id \otimes \e ) \phi '
\simeq \id$ and $(\e \otimes \id) \phi' \simeq \id$.
\end{prop}

\begin{proof} Since $P_0 ^*, \dots, P_{n-2}^*$ are
projective and $\bD _{\z} \otimes \bD _{\z}$ has no homology in dimensions $0<i<n-1$, the map $\phi : \Uz \to \Uz
\otimes \Uz$ extends to a chain map $\phi' : \bD _{\z} \to
\bD _{\z} \otimes \bD _{\z}$. Since $(\id \otimes \e ) \phi
= \id=(\e \otimes \id) \phi$, we can construct homotopies $
(\id \otimes \e ) \phi ' \simeq \id$ and $(\e\otimes \id)
\phi' \simeq \id $.
\end{proof}

Now, we consider the chain complex $\bP (\z)$ defined in the introduction. Recall that $\bP (\z)$ is a chain complex of projective modules that fits into an extension of the form 
\begin{equation}\label{eqn:defining sequence for Lz}
0 \to \Sig^{n-1} \bP \to \bP (\z) \to \bP \to 0
\end{equation}
with extension class $\z \in [\bP , \Sig ^n \bP ]$. Our
first observation is the following:

\begin{prop}\label{prop:Projective resolution}
The complex $\bP (\z)$ is a projective resolution of
$\bC _{\z}$, i.e., there is a chain map $\bP (\z) \to \bC_
{\z}$ that induces an isomorphism on homology.
\end{prop}

\begin{proof} The proof follows from an argument given in the proof of Lemma 3.1 in \cite{Paper:BensonCarlson94ProjectiveResolutions}. 
Let $\bC_{\z} ^{\infty}$ denote the complex obtained by splicing $\bC_{\z}$ with itself infinitely many times in the positive direction. Note that there is a short exact sequence of the form $$ 0 \to \bC _{\z} \to \bC _{\z} ^{\infty} \to \Sig ^n \bC _{\z}^{\infty} \to 0.$$  After tensoring this sequence with a projective resolution $\bP$ of $k$, we obtain a short sequence of projective chain complexes of the form
$$ 0 \to \bP \otimes \bC _{\z} \to \bP \otimes \bC _{\z} ^{\infty} \to \bP \otimes \Sig ^n \bC _{\z} ^{\infty} \to 0.$$ The complex $\bP \otimes \bC _{\z} ^{\infty}$ is a projective resolution of $k$ since the complex $\bC _{\z} ^{\infty}$ has the homology of a point. Similarly, the complex $\bP \ot \Sig ^n \bC _{\z}^{\infty}$ is homotopy equivalent to $\Sig ^n \bP$. It is shown in \cite[page 455]{Paper:BensonCarlson94ProjectiveResolutions} that the
map $$ \bP \otimes \bC _{\z} ^{\infty} \to \bP \otimes \Sig ^n \bC_{\z} ^{\infty}$$ represents the cohomology class $\z$ under the identification $[\bP, \Sig ^n \bP ]=\Ext_{kG} ^n (k,k)$.  Now, by Lemma \ref{lem:kernel and extensions}, we can conclude that $\bP \otimes \bC _{\z}$ is chain homotopy equivalent to $\bP (\z)$. So there is a chain map $\bP (\z) \to \bC _{\z}$ which induces an isomorphism on homology.
\end{proof}

We have the following immediate corollary.

\begin{cor}\label{cor:Projective resolution}
The complex $\bP (\z)$ is a projective resolution
of $\bD _{\z}$.
\end{cor}

\begin{proof}
This follows from the fact that there is a chain map $\bC _{\z} \to \bD _{\z}$ which induces an isomorphism on homology (see \cite[Proposition 5.2]{Paper:BensonCarlson94ProjectiveResolutions}).
\end{proof}

Now we are ready to prove Theorem \ref{thm:main1}.

\begin{proof}[Proof of Theorem \ref{thm:main1}]
Assume that $\z$ is productive. Then, by Proposition \ref{prop:slitting for Cz}, there is a chain map $\phi' : \bD _{\z} \to \bD _{\z} \otimes \bD_{\z} $ such that $(\id \otimes \e ) \phi' \simeq \id$ and $ (\e\otimes \id) \phi' \simeq \id$. By Corollary \ref{cor:Projective resolution},
$\bP (\z)$ is a projective resolution for $\bD _{\z}$. So, by the standard properties of projective resolutions, there is a chain map
$\D:  \bP (\zeta ) \to \bP (\z) \otimes \bP (\z)$ which
makes the following diagram commute:
\begin{equation*}
\begin{CD}
\bP (\z) @>>> \bD _{\z} \\
@V{\D}VV @VV{\phi'}V \\
\bP (\z) \otimes \bP (\z) @>>> \bD _{\z} \otimes \bD _{\z}\ . \\
\end{CD}
\end{equation*}
Since both $(\id \otimes \e ) \D $ and $(\e \otimes \id ) \D$ induces the identity map on homology, they are homotopic to the identity map on $\bP (\z)$.

For the converse, assume that there is a chain map $\Delta : \bP (\z) \to \bP (\z) \otimes \bP (\z)$ satisfying the properties. Consider the sequence
$$0 \to \Sig ^{n-1} \bP \otimes \bP (\z) \to \bP (\z) \otimes \bP(\z) \to \bP \otimes \bP (\z) \to 0$$ which is obtained by tensoring the sequence (\ref{eqn:defining sequence for Lz}) with $\bP (\z)$. Note that using the chain homotopy equivalence $\bP \otimes \bP (\z) \simeq k \otimes \bP (\z) \cong \bP (\z)$, we can view the diagonal map $\D$ as a splitting map for this exact sequence. This implies that the composition
\begin{equation*} \begin{CD} \mu_{\z}: \bP (\z)
@>\simeq>> \bP \otimes \bP (\z) @>{\hat \z \otimes \id}>> \Sig ^n \bP\otimes \bP (\z) @>\simeq>> \Sig^n \bP (\z) \\
\end{CD}
\end{equation*}
is homotopic to zero.

Given a chain complex $\bC$, let $\Gamma _k \bC$ denote the truncation of $\bC$ at $k$. This is a complex where $(\Gamma _k \bC)_i =\bC_i$ for all $i\geq k$ and $(\Gamma _k \bC)_i =0$ otherwise. The differentials of $\Gamma _k \bC$ are the same as the differentials of $\bC$ whenever it is not zero. Note that when we truncate both of the complexes $\bP (\z)$ and $\Sig^n \bP(\z)$ at $k=2n$, we get a  chain map $$\Gamma _{2n} \mu_{ \z} : \Gamma _{2n} \bP (\z) \to \Gamma _{2n} \Sig ^n \bP (\z).$$  Both of these truncated complexes have only one nontrivial homology which is at dimension $2n$. It is easy to see that $$H_{2n} (\Gamma _{2n} \bP (\z)) \cong \Om ^n \Lz \oplus (\proj) \ \ \text{and}\ \ H_{2n} (\Gamma _{2n} \Sig ^n \bP (\z ))\cong \Lz \oplus (\proj)$$ and we claim that the map induced by $\Gamma _{2n} \mu _{\z}$ on homology is stably equivalent to the composition
\begin{equation*}\begin{CD} \Om ^n \Lz @>\cong>>  \Om ^n
k  \otimes \Lz @>{\hat \z \otimes \id}>> k \otimes \Lz  @>\cong>> \Lz \ .\end{CD}
\end{equation*} 
To see this first note that $$H_{2n}(\Gamma _{2n} (\bP \otimes \bPz ))\cong H_{2n} (\Gamma _{n} \bP \otimes \Gamma _n \bPz)\oplus Q $$ for some projective module $Q$. This gives that $$H_{2n} (\Gamma _{2n} (\bP \otimes \bPz))\cong (\Omega ^n k \otimes \Lz) \oplus Q.$$
Similarly, we have $$H_{2n} (\Gamma _{2n} (\Sig ^n \bP \otimes \bPz ))\cong (k \otimes \Lz)\oplus Q'$$ for some projective module $Q'$ and the map between the nonprojective parts of these modules is induced by the chain map $$\Gamma _n \hat \zeta \otimes \id:\Gamma _n \bP \otimes \Gamma _n \bPz \to \Gamma _n \Sig ^n \bP \otimes \Gamma _n \bPz.$$
The map induced by this chain map on homology is obviously
 \begin{equation*}\begin{CD} \Om ^n
k  \otimes \Lz @>{\hat \z \otimes \id}>> k \otimes \Lz \end{CD}
\end{equation*} so the claim is true.

Now, since $\mu_{\z} \simeq 0$, the map induced by the chain map $\Gamma _{2n} \mu _{\z}$ on homology splits through a projective module. It follows that the image of 
$\z$ in $$\underline \Hom _{kG} (\Omega ^n \Lz, \Lz )\cong \Ext ^n_{kG} (\Lz, \Lz )$$ is zero.
\end{proof}

\section{Proof of Theorem \ref{thm:main2}}
\label{sect:main2}

In this section, we study the obstructions for the existence of a chain map $\D: \bPz \to \bPz \otimes \bPz$ satisfying $(\e \ot \id)\D\simeq \id$ and $(\id \ot \e )\D\simeq \id$. The chain complex $\bPz$ fits into an 
extension of the form $$ 0 \to \Sig ^{n-1} \bP \to \bPz \to \bP \to 0 $$ with extension class $\z \in [\bP, \Sig ^n \bP]$. To avoid complicated formulas with $(-1)^{n-1}\bd$, we regard $\Sig ^{n-1} \bP$ not as a shift 
of $\bP$, but as a separate chain complex denoted by $\bQ$. Let $\al: \bP \to \Sig \bQ$ be a representative of the extension class of this extension. Then, by choosing a $kG$-module splitting, we can express the differential on $\bP (\z)$ as a matrix by $$\bd =\left[\begin{matrix}
\bd ^Q & \al \\ 0 & \bd ^P\end{matrix}\right].$$

The splitting for $\bPz$ gives a splitting for the complex $\bPz \otimes \bPz$ 
where $$(\bPz \otimes \bPz )_i= (\bQ\otimes \bQ)_i \oplus (\bQ \otimes \bP)_i 
\oplus ( \bP \otimes \bQ )_i \oplus (\bP\otimes \bP)_i$$ for all $i$ and with respect to this splitting, the differential for $\bPz \ot \bPz$ can be expressed in the matrix form as $$\bd =\left[\begin{matrix}\bd & 
\id \times \al & \al \times \id & 0 \\ 0 &  \bd  & 0 & \al \times \id \\ 
0 & 0 & \bd  & \id \times \al \\ 0 & 0 & 0 & \bd
\end{matrix}\right].$$
Note that the differentials on the diagonal of the above matrix are of the form $$
\bd=\bd \times \id +\id \times \bd$$ and, by the usual
convention of signs, we have $$(\bd \times \id+\id \times
\bd)(x\otimes y)=\bd (x)\ot y+(-1)^{\deg x} x \otimes \bd (y).$$ Note also that $$(\id \times
\al)(\al \times \id)=(-1)^{\deg \al} (\al \times \id)(\id
\times \al)=-(\al \times \id)(\id \times \al),$$ so the
above matrix squares to zero.

Because of the shape of the matrix for $\bd$, there is a 
3-step filtration for $\bPz \otimes \bPz$. Let us define
$\bP (\z \oplus \z)$ as the chain complex
$$\bP (\z \oplus \z) _i= (\bQ \otimes \bP)_i \oplus (
\bP \otimes \bQ)_i \oplus (\bP \otimes \bP)_i$$ with differential $$\bd =\left[\begin{matrix}
\bd & 0 & \al \times \id \\ 0 & \bd & \id \times \al \\ 0 & 0 & \bd \end{matrix}\right].$$
Note that $\bP (\z \oplus \z)$ fits into the extension of the form
$$0 \to (\bQ \otimes \bP) \oplus (\bP \otimes \bQ)
\to \bP (\z \oplus \z ) \to \bP \otimes \bP \to 0 $$
with extension class
$$\theta =\left[\begin{matrix} \al \times \id \\
\id \times \al \end{matrix}\right].$$
By our choice of $\bP (\z \oplus \z )$, there is also an extension
$$ 0 \to \bQ \otimes \bQ \to \bPz \otimes \bPz \to
\bP (\z \oplus \z) \to 0$$ with extension class
$$ \eta=\left[\begin{matrix} \id \times \al & \al \times \id & 0\end{matrix}\right].$$
Our first observation is the following:

\begin{prop}\label{prop:firstlifting}
There is a chain map $\psi : \bPz \to \bP (\z \oplus \z)$
which makes the following diagram commute
\begin{equation*}\begin{CD} 0 @>>> \bQ @>>> \bPz @>>> \bP
@>>> 0 \\ @. @VV{(\D _1, \D _2 )}V @VV{\psi}V
@VV{\D }V @. \\
0 @>>> (\bQ \otimes \bP) \oplus (\bP \otimes \bQ ) @>>> \bP
(\z \oplus \z ) @>>> \bP \otimes \bP @>>> 0 \\
\end{CD}
\end{equation*}
where $\D , \D_1, \D_2$ are chain maps covering the map $k\to k\otimes k$ defined by $\lambda \to \lambda \otimes 1$.
\end{prop}

This is in some sense saying that there are no obstructions
for lifting $\D : \bP \to \bP \otimes \bP$ to a chain map $
\bP (\z) \to \bP (\z \oplus \z )$. To prove Proposition \ref{prop:firstlifting},
first observe that if $\psi$ exists, then it must be of the
form $$\psi=\left[\begin{matrix} \D _1  & H_1 \\
\D _2 & H_2 \\ 0 & \D  \end{matrix}\right]$$ where $H_1$ and $H_2$ satisfy the following formulas:
\begin{equation}\label{eqn:H1 and H2}
\begin{split} \bd H_1 - H_1 \bd & = \D _1 \al- (\al \times \id) \D \\
\bd H_2 - H_2 \bd & = \D_2 \al - (\id \times \al ) \D \ .\\
\end{split}
\end{equation}

Note that maps on the right hand side are of the form $\bP
\to \Sig (\bQ \ot \bP)$ or $\bP \to \Sig (\bP\ot \bQ)$, so
Proposition \ref{prop:firstlifting} follows from the
following lemma.

\begin{lem}\label{lem:homotopy equivalences}
There are homotopy equivalences  $(\al \times \id) \D \simeq \D_1 \al$ and $(\id \times \al ) \D \simeq  \D _2 \al$.
\end{lem}

\begin{proof} We will show that $(\al \times \id)\D \simeq
\D_1 \al$. A proof for the second homotopy equivalence can
be given in a similar way. Let $\e : \bP \to k$ denote the
augmentation map and let $\id \otimes \e : \bA \otimes \bP \to \bA$ denote the chain map defined by $(\id \otimes \e)(a \otimes b)=\e(b)a$ where $\bA=\bP$ or $\bQ$. Then, we have $$(\id \otimes \e ) (\al \times \id ) \D \simeq \al (\id \otimes \e)\D \simeq \al \simeq (\id \ot \e ) \D_1 \al.$$ Since $\id \otimes \e$ is a homotopy equivalence, the result follows.
\end{proof}

Note that there is more than one chain map $\psi$ that fits into the diagram given in Proposition 
\ref{prop:firstlifting} depending on the choices we make for homotopies $H_1$ and $H_2$.  When we want to emphasize the dependency of $\psi$ on $H_1$ and $H_2$, we will use the notation $\psi (H_1, H_2)$. We will see later that the
answer to the question whether or not $\psi (H_1, H_2)$ lifts to a chain map $\tilde \psi : \bPz \to \bPz \ot \bPz$ does not depend on $H_1$ and $H_2$.

Observe that if $(H_1, H_2)$ and $(H_1', H_2')$ are two different choices of homotopies satisfying the equations in (\ref{eqn:H1 and H2}), then the differences
$H_1-H_1'$ and $H_2-H_2'$ are chain maps. We will see below that, up to chain maps, the homotopies $H_1$ and $H_2$ can be chosen to satisfy certain identities. For chain complexes $\bA$ and $\bB$, let $T: \bA \otimes \bB \to \bB \otimes \bA$ be the chain map defined by $$T(a\ot b)= (-1)^{\deg(a)\deg(b)} b\ot a.$$ Observe that for maps $f \in \bHom (\bC, \bD)$ and $g \in \bHom (\bE, \bF)$, we have
$$T(f\times g)T=(-1)^{\deg(f)\deg(g)} g \times f.$$ In particular, we have $(\id \times \al )= T(\al \times \id)T$.

Since $T: \bP \otimes \bP \to \bP \otimes \bP$ induces the identity map on homology, it is homotopic to the identity map. Let $H$ denote the homotopy that satisfies $\bd H +H\bd=\id-T$. Similarly, there is a homotopy between $\D_2$ and $T\D_1$. Let $H'$ be the homotopy $\bd H'+H'\bd=\D_2-T\D_1 $. We obtain the following:

\begin{lem}\label{lem:Relation between homotopies H1 and H2} Up to chain maps we can choose the homotopies $H_1$ and $H_2$ so that they satisfy the relation $$H_2-TH_1=H' \al +(\id \times \al)H \D.$$
\end{lem}

\begin{proof} Applying $T$ to the first equation in (\ref{eqn:H1 and H2}), we get
\begin{equation*}
\begin{split} (\bd H_2-H_2\bd) -T(\bd H_1-H_1\bd) &= \D_2 \al- (\id \times \al)\D -T\D_1\al + T(\al \times \id) \D \\
&=(\D _2 -T\D_1)\al- (\id \times \al) (\id -T)\D \\
&=(\bd H'+H'\bd)\al -(\id \times \al) (\bd H +H\bd)\D\\
&= \bd H'\al-H'\al\bd +\bd(\id \times \al )H\D-(\id \times \al )H\D\bd .\\
\end{split}
\end{equation*}
\end{proof}

Now, we are interested in finding the obstructions for
lifting the map $\psi : \bPz \to \bP (\z \oplus \z)$ to a
chain map $\tilde \psi : \bPz \to \bPz \otimes \bPz$ so that $ \pi \tilde \psi =\psi$ where $\pi$ is the map in the extension (\ref{eqn:postnikov1}). Since the extension class
of this extension is $$ \eta=\left[\begin{matrix} \id \times \al & \al \times \id & 0\end{matrix}\right],$$ by Lemma \ref{lem:obstruction for lifting of maps}, there is a unique obstruction for lifting $\psi$ which is the homotopy class of the composition $\eta \psi$.  The following is true for this obstruction:

\begin{prop}\label{pro:main2} Let $\psi (H_1,H_2)$ be a chain map that fits into the diagram given in Proposition
\ref{prop:firstlifting} and let $\eta$ be as above. Suppose that $H_1$ and $H_2$ satisfy the relation given in Lemma \ref{lem:Relation between homotopies H1 and H2}. Then, the following holds:
\begin{enumerate}
\item If $p$ is odd and $n$ is even, then $\eta \psi \simeq 0$.
\item  If $p=2$, then $\eta \psi \simeq 0$ if and only if
$\widetilde{Sq}^{n-1} \z$ is a multiple of $\z$.
\end{enumerate}
\end{prop}

\begin{proof} By Lemma \ref{lem:usefulobservation}, the map $\eta \psi : \bPz \to \Sig (\bQ \ot \bQ)$ is homotopic zero if and only if its composition with
$$\e \ot \e: \Sig (\bQ \ot \bQ) \to \Sig ^{2n-1}\bk $$ is homotopic to zero where $\e \ot \e$ is the map defined by $(\e \ot \e )(a\ot b)=\e(a)\e(b)$.
The composition $\varphi=(\e \ot \e)\eta \psi$ can be expressed as
a matrix $\varphi=[\varphi_1 \ \varphi_2]$ where
\begin{equation*}
\begin{split}
\varphi_1 & = (\e \ot \e)(\id \times \al)
\D_1 +(\e \ot \e)(\al \times \id)\D_2 \\ \varphi _2 & =(\e \ot \e) (\id \times \al)H_1 +(\e \ot \e) (\al \times \id)
H_2.
\end{split}
\end{equation*}
Note that $(\e \ot \e)T=(-1)^{n-1}(\e
\ot \e)$, so we have $$ \varphi_1 \simeq(\e \ot \e)(\id \times \al)\D_1 +(\e \ot \e) (\al \times \id)T\D_1=(1+(-1)^{n-1} ) (\e \ot \e)(\id \times \al)\D_1.$$
Thus if $p=2$, or $p>2$ and $n$ is even, then $\varphi _1\simeq 0$. The homotopy between $\varphi_1$ and the zero map can be taken as the composition $G=(\e \ot \e)(\al \times \id)H'$ where $H'$ is the homotopy satisfying 
$\bd H'+H'\bd =\D _2-T\D_1$.

By Lemma \ref{lem:maps from the middle}, we have $\varphi =[\varphi_1 \
\varphi _2]\simeq [0 \ \varphi _2 ' ]$ where $\varphi_2 '= \varphi_1-G\al.$
Assuming that $p=2$, or $p>2$ and $n$ is even, and using the relation given in Lemma \ref{lem:Relation between homotopies H1 and H2}, we can simplify $\varphi_2'$ as follows:
\begin{equation}\label{eqn:psiprime}\begin{split}\varphi_2' & =(\e \ot \e)(\id \times \al)
H_1 +(\e \ot \e) (\al \times \id) H_2-(\e \ot \e)(\al \times \id)H'\al \\
&=(\e \ot \e)(\al \times \id) [-TH_1 +H_2-H'\al] \\
&=(\e \ot \e)(\al \times \al)H\D.
\end{split}
\end{equation}
To complete the proof we need to show that
$$(\e \ot \e)(\al \times \al)H\D: \bP \to \Sig ^{2n-1} \bk
$$ is homotopy equivalent to a chain map of the form
$(\Sig ^n u)\al$ where $u: \bP \to \Sig ^{n-1} \bk$. Then the result will follow from Lemma \ref{lem:maps from the middle}. Note that if $\hat \z : \bP \to \Sig ^n \bk$ is the chain map associated to $\z$, then we have $$(\e \ot \e)(\al \times \al)H\D= (\hat \z \otimes \hat \z )H \D.$$

Since $H\D$ is a homotopy between $\D$ and $T\D$, if $k=\bF_2$, then the cohomology class associated to this chain map is the Steenrod square $Sq ^{n-1} \z$ by the classical definition of Steenrod squares over $\bF _2$ (see \cite[page 272]{Book:Spanier1966AlgTop}). For an arbitrary field $k$ of characteristic $2$, we need to take the $Sq ^{n-1}$ action on $H^*(G,k)$ as the semilinear extension of
$Sq^{n-1}$ action on $H^*(G,\bF _2)$ as defined in the introduction. The reason for taking semilinear extension rather than the usual linear extension is that if we multiply $\z$ with some $\lambda \in k$,
then the homotopy class of the chain map $(\hat \z \ot \hat \z) H\D$ is multiplied by $\lambda ^2$. This can be easily seen by using the bar resolution and taking a specific homotopy for $H\D$ (see \cite[page 142]{Book:Benson95Rep&CohII}.)

For $p>2$, observe that the chain complex $\bP \ot \bP$ decomposes as $\bP \ot \bP=\bD_+ \oplus \bD_-$ where $$\bD_+=(\id+T)(\bP \ot \bP) \ \ {\rm and } \ \ \bD_-= (\id-T) (\bP \ot \bP ).$$ Note that $\bD_-$ has zero homology, so there is a contracting homotopy $s: \bD_- \to \Sig ^{-1} \bD _-$. Using this homotopy, we can choose the
homotopy $H$ between $\id$ and $T$ as the composition $s (\id -T) : \bP \ot \bP \to \bP \ot \bP$. But then the image of $H$ will be in $\bD_-$ and we will have $$(\e \ot \e)(\al \ot \al)H\D=0$$ because the composition $(\e \ot \e)(\al \ot \al)(\id-T)$ is equal to zero.
\end{proof}

Note that although we made a specific choice for $(H_1, H_2)$ in the proposition above, the same conclusion holds for every choice of homotopies $H_1$ and $H_2$. This follows from the following proposition:

\begin{prop}\label{pro:lfting for any choice} If\ $\psi (H_1, H_2)$ lifts
to a chain map $\tilde \psi$ satisfying $\pi \tilde \psi=\psi$  for
some choice of homotopies $H_1$ and $H_2$, then $\psi (H_1', H_2')$
lifts to a chain map $\tilde \psi$ satisfying $\pi \tilde \psi=\psi$
for any other choice of homotopies $(H_1', H_2')$.
\end{prop}

\begin{proof} If $H_1$ and $H_2$ are replaced with $H_1'$ and
$H_2'$, then $H_1'=H_1+f_1$ and $H_2'=H_2+f_2$ for some chain maps
$f_1$ and $f_2$. In Equation \ref{eqn:psiprime}, if we replace
$(H_1, H_2)$ with $(H_1', H_2')$, then the difference between the
old $\varphi '_2=(\e\ot\e)(\al \times \al)H\D$ and the new one would
be
$$(\e\ot \e)(\id \times \al)f_1+ (\e\ot \e)(\al \times \id)f_2.$$
It is clear that the homotopy class of this map is a multiple of
$\al$. So, the new $\psi_2 '$ is a multiple of $\al$ if and only if
the old one is.
\end{proof}

Now to complete the proof of Theorem \ref{thm:main2} we prove the
following:

\begin{prop}\label{pro:lifting for all}
For some (every) choice of $H_1$ and $H_2$, the chain map
$\psi (H_1, H_2)$ lifts to a chain map $\tilde \psi$ if and only if
there is a chain map $\D: \bPz \to \bPz \otimes \bPz$ satisfying
$(\id \otimes \e)\D \simeq \id $ and $(\e \otimes \id)\D \simeq
\id$.
\end{prop}

\begin{proof} First suppose that there is a lifting $\tilde \psi$.
Then, we take $\D$ as $\tilde \psi$ and show that it satisfies the
required conditions. We will only show that $(\id \otimes \e)\D
\simeq \id $. The second homotopy equivalence can be shown in a
similar way.

Since the restriction of $\id \otimes \e$ on $\bQ \ot \bQ$ is the
zero map, we have
$$(\id \ot \e)\tilde \psi=(\id \ot \e)\psi =\left[\begin{matrix}
\id \ot \e &0 &0 \\
0 & 0  & \id \ot \e \end{matrix}\right] \left[\begin{matrix} \D _1  & H_1 \\
\D _2 & H_2 \\ 0 & \D  \end{matrix}\right] =\left[\begin{matrix} (\id\ot \e)\D _1  & (\id \ot \e) H_1 \\
0 & (\id \ot \e) \D \end{matrix}\right].$$ Both $(\id \ot 
\e)\D_1$ and $(\id \ot \e)\D$ are homotopic to identity 
maps. In fact, by choosing $\bP$ as the bar resolution and 
$\D$ as the diagonal approximation given by $$ \D([g_1, \dots , g_n])=\sum _{i=0} ^n [g_1, \dots , g_i]\otimes (g_1\cdots g_i)[g_{i+1},\dots, g_n]$$ we can assume that $(\id \ot \e)\D=\id$. Similarly, we can choose a $\D_1$ such that $(\id \ot \e)\D_1=\id$.

Now, it is easy to see that $(\id \ot \e) H_1$ is a chain map and $(\id \ot \e)\tilde \psi \simeq \id$ if and only if $(\id \ot \e) H_1$ is homotopic to a map of the form $(\Sig ^n u)\alpha $ for some $u$. By Proposition \ref{pro:lifting for all}, we can replace $H_1$ with another homotopy up to a chain map. Replacing $H_1$ with $H_1'=H_1-\D (\id \ot \e )H_1$, we get $(\id \ot \e)H_1' \simeq 0$. This gives $(\id \ot \e)\tilde \psi \simeq \id$ as desired.

For the converse, assume that there is a chain map $\D : \bPz \to
\bPz \ot \bPz$ satisfying $(\id \otimes \e)\D \simeq \id $ and $(\e
\otimes \id)\D \simeq \id$. We need to show that the composition
$\pi \D$ is homotopy equivalent to $\psi (H_1, H_2)$ for some
$H_1$ and $H_2$. For this, it is enough to show that $\pi \D$
fits into a diagram as in Proposition \ref{prop:firstlifting}. Note that $(\id \ot \e ) \pi \D\simeq \id $, so there exists $f_Q$ and $f_P$ which makes the following diagram commute
\begin{equation*}\begin{CD} 0 @>>> \bQ @>>> \bPz @>>> \bP
@>>> 0 \\ @. @VV{f_Q}V @VV{\pi \D}V @VV{f_P}V @. \\
0 @>>> (\bQ \otimes \bP) \oplus (\bP \otimes \bQ ) @>>> \bP (\z \oplus \z ) @>>> \bP \otimes \bP @>>> 0
\\ @. @VV{\id \ot \e}V @VV{\id \ot \e}V @VV{\id \ot \e}V @. \\ 0 @>>> \bQ  @>>> \bPz @>>> \bP  @>>> 0 \\
\end{CD}
\end{equation*}
Since $\id \ot \e$ induces homotopy equivalences $\bQ \ot \bP \simeq \bQ$ and $\bP \ot \bP \simeq \bP$, we have $f_P \simeq \D$ and $f_Q$ is homotopic to $\D_1$ when it is composed with the projection to $\bQ \ot \bP$. Repeating this argument also for $\e \ot \id$, we get $f_Q\simeq (\D_1, \D_2)$. This completes the proof.
\end{proof}

In the case where $k$ is an arbitrary field of characteristic $2$, the Steenrod squares in Theorem \ref{thm:main2}  are semilinear
extensions of usual Steenrod squares over $\bF _2$. To illustrate
the importance of this point, we give the following example.

\begin{ex}\label{ex:semilinear extension} Let $G=\bbZ /2 \times \bbZ /2$
and $\{ x,y \}$ be a basis for $H^1(G, \bbF _2)$. Take $\z
=x+\lambda y$ for some $\lambda \in k$. A direct computation with
$\Lz$-modules shows that $\z$ is productive if and only if $\lambda
\in \bbF _2$. Now, we can also see this by Theorem \ref{thm:main2}.
Since $\widetilde{Sq}^0 (x+\lambda y)=x+\lambda ^2 y$, we have $\widetilde{Sq}^0 (\zeta)\in
(\zeta)$ if and only if $\lambda ^2=\lambda$. In fact, in this case
one can explicitly write down a few steps of the homotopy between
$\D$ and $T\D$ in the bar resolution and see that if $\hat \z:
\bP_1\to k$ is a map representing $\z$, then $\widetilde{Sq}^0 (\z)$ is
represented by the composition $$\bP_1 \to \bP_1 \otimes \bP_1
\maprt{\hat \z \otimes \hat \z } k$$ where the first map is given by
$[g]\to [g]\otimes [g]$ on the standard basis of the bar resolution.
\end{ex}

\section{Semi-productive elements}
\label{sect:semi-productive elements}

In this section, we introduce the notion of semi-productive elements and prove Theorem \ref{thm:main3} stated in the introduction. The arguments used in this section are very similar to the arguments used in \cite[Section
11]{Paper:BensonCarlson94ProjectiveResolutions}.

Let $k$ be a field of characteristic $p>0$ and let $\z \in H^n (G,k)$ be a nonzero cohomology class where $n\geq 1$. We say $\z$ is {\it semi-productive} if $\z$ annihilates $\Ext^* _{kG} (\Lz, k)$. Recall that $\z$ acts on $\Ext_{kG} ^*(\Lz , k)$ via the map $$\Ext_{kG}^p (k,k)\otimes \Ext_{kG} ^q (\Lz , k) \to \Ext _{kG}^{p+q} (\Lz, k)$$ which can be defined in various ways, one of which is the Yoneda splice over $k$. Note that in this case, the Yoneda
splice coincides with the outer product, hence up to a sign, this product is the same as the product defined by first tensoring the extension for $\z$ by $\Lz$ and then splicing it over $\Lz$ (see \cite[Section 6]{Book:Carlson1996Modules} for more details). So, we
can conclude the following:

\begin{prop}\label{pro:prod implies semiprod}
If $\z$ is a productive element, then it is
semi-productive.
\end{prop}

If $p$ is odd and the degree of $\z$ is even, then $\z$ is productive, hence it is semi-productive. So there is nothing to study when $p$ is odd. Therefore, from now on, we assume that $k$ is a field of characteristic $2$. 

Note that the converse of Proposition \ref{pro:prod implies semiprod} is not true in general. We will show later that if a cohomology class is a nonzero divisor, then it is semi-productive. This allows us to give examples of semi-productive elements which are not productive (see Example \ref{ex:semi but not prod}).

For studying semi-productive elements, the following commuting diagram is very useful \Small
\begin{equation*}
\begin{CD}
@>{\hat \z ^*}>> \Ext^{i+n} _{kG} (k,k) @>{j^*}>> \Ext^i _{kG} (\Lz, k) @>{\delta}>> \Ext^{i+1} _{kG} (k,k) @>{\hat \z ^*}>> \Ext^{i+n+1}_{kG} (k,k) @>>> \\
@. @VV{\z \cdot}V @VV{\z \cdot}V @VV{\z \cdot}V @VV{\z \cdot}V \\
@>{\hat \z ^*}>> \Ext^{i+2n} _{kG} (k,k) @>{j^*}>> \Ext^{i+n} _{kG} (\Lz, k) @>{\delta}>> \Ext^{i+n+1} _{kG} (k,k) @>{\hat \z ^*}>> \Ext^{i+2n+1}_{kG} (k,k) @>>> \\
\end{CD}
\end{equation*}
\normalsize
where the top and the bottom row comes from the short
exact sequence
$$ 0 \to \Lz \maprt{j} \Om ^n k \maprt{\hat \z} k \to  0$$ and the vertical maps are given by multiplication by $\z$. Note that $\hat \z ^*$ can also be expressed as  multiplication by $\z$, i.e., we have $\hat \z^* (x)=\z x$ for every $x\in \Ext ^i _{kG} (k,k)$. 

Observe that for every $x\in \Ext^{i+n} _{kG} (k,k)$, we have  $\z j^* (x) =j^* (\z x) =j^* \hat \z ^* (x)=0$. So, for every $u \in \Ext^i _{kG} (\Lz, k)$, the product $\z u$ is uniquely determined by $\delta(u)$. Also note that if $u \in \Ext^i _{kG} (\Lz, k)$, then $\delta (\z u)=\z \delta (u)=\hat \z ^* \delta (u)=0.$ This means that there is a $y \in \Ext^{i+2n} _{kG} (k,k)$
such that $\z u=j^*(y)$. This element $y$ is uniquely defined modulo
the ideal $(\z)$ generated by $\z$. Hence, we can conclude that for all $i\geq 0$, the map $$\Ext^i _{kG} (\Lz, k) \maprt{\z \cdot} \Ext^{i+n} _{kG} (\Lz, k)$$ induces a $k$-linear map $$ \mu  : \Ann ^{i+1} (\zeta) \to \Ext^{i+2n} _{kG} (k,k)/ (\z)$$ where $\Ann ^{i+1} (\zeta)$ is the subspace  formed by elements $v \in \Ext ^{i+1}_{kG} (k,k)$ such that $\z v=0$. By the definition of this map, we
have the following:

\begin{lem}\label{lem:obstruction mu} Let $\z \in H^n (G, k)$ be a nonzero cohomology class of degree $n$ where $n \geq 1$. Then, $\z$ is semi-productive if and only if $\mu (v)=0$ for every $v \in \Ext^{i+1} _{kG} (k,k)$ satisfying $\z v=0$ where $i\geq 0$.
\end{lem}

In the rest of the section, we analyze the obstructions $\mu (v)$ and relate it to Massey products and then to Steenrod operations. We first recall the definition of a triple Massey product.

Let $u,v,w \in H^*(G, k)$ be homogeneous cohomology classes and let $\hat u: \bP \to \Sig ^r \bP$, $\hat v: \bP\to \Sig ^s \bP$, and $\hat w : \bP \to \Sig ^t \bP$ be chain maps whose homotopy classes are $u,v,w$, respectively. Suppose that $u v=0$ and $vw=0$. Then there exist homotopies $H:\bP \to \Sig ^{r+s-1} \bP$ and $K: \bP \to \Sig ^{s+t-1} \bP$ satisfying $\bd H+H\bd =\hat u \hat v$ and $\bd K +K \bd =\hat v \hat w$. These equations give that
$H\hat w+\hat uK$ is a chain map 
$\bP \to \Sig ^{r+s+t-1} \bP$, so it defines a cohomology class in $H^{r+s+t-1} (G, k)$. This cohomology class is well-defined modulo the subspace $u H^{s+t-1} (G,k)+ wH^{r+s-1} (G,k)$ generated by $u$ and $w$. The triple Massey product $\la u,v,w\ra $ is defined as the set of homotopy classes of chain maps 
$H\hat w+\hat uK$ over all possible choices of $H$ and $K$. Alternatively, one can consider the triple Massey product as an equivalence class and denote by 
$\la u,v,w \ra$ a representative of this equivalence class. We use this second approach here in this paper.

Given  $u,v,w \in H^*(G, k)$ as above, let $\bP (u)$ and $\bP (w)$ denote the extensions with extension classes $u$ and $v$, respectively. We have a diagram of the following form \Small
\begin{equation*}
\begin{CD}
@>{w\cdot}>> [\bP , \Sig^{s+t-1} \bP] @>{\pi^*}>> [\bP (w), \Sig ^{s+t-1} \bP] @>{j ^*}>> [\bP, \Sig ^s \bP] @>{w\cdot}>> [\bP, \Sig ^{s+t}\bP] @>>> \\
@. @VV{u\cdot}V @VV{u\cdot}V @VV{u\cdot}V @VV{u \cdot}V \\
@>{w \cdot}>> [\bP , \Sig^{r+s+t-1} \bP] @>{\pi^*}>> [\bP (w), \Sig^{r+s+t-1} \bP] @>{j^*}>> [\bP, \Sig ^{r+s} \bP] @>{w \cdot }>> [\bP, \Sig ^{r+s+t}\bP] @>>> \\
\end{CD}
\end{equation*}
\normalsize where the horizontal sequences comes from the extension
$$ 0 \to \Sig ^{t-1} \bP \maprt{j}\bP (w) \maprt{\pi} \bP \to 0.$$
A diagram chase similar to the one used above shows that for every $v\in [\bP , \Sig ^s \bP]$ satisfying $uv=0=vw$, there is a class $x \in [\bP , \Sig ^{r+s+t-1} \bP]$ well-defined modulo
$$J(u,w):=u [\bP , \Sig^{s+t-1} \bP ] +w [\bP , \Sig^{r+s-1} \bP]$$ such that $\pi^*(x)=uy$ where $y \in [\bP (u), \Sig ^{s+t-1} \bP]$ is a class satisfying $j^*(y)=v$. We have the following:

\begin{lem}
The cohomology class $x\in [\bP , \Sig ^{r+s+t-1} \bP]$ can be taken as the triple Massey product $\la u, v, w\ra$ modulo $J(u,w)$.
\end{lem}

\begin{proof} Let $K$ be a contracting homotopy for $\hat v\hat w$. Then, we can take $y\in [\bP(w), \Sig ^{s+t-1} \bP]$ as the homotopy class of the chain map given by $\hat y =[\hat v \ K]$. This means that $uy$ is represented by 
$\hat u \hat y =[\hat u \hat v \ \hat u K]$. Let $H$ be a contracting homotopy for $\hat u \hat v$. Then, by Lemma \ref{lem:maps from the middle}, we have
$\hat u \hat y \simeq [0 \ \hat uK+H\hat w].$ So, $x$ can be taken as the homotopy class of the chain map $\hat uK +H\hat w$. Hence, $x \equiv \la
u, v,w \ra$ modulo $J(v,w)$.
\end{proof}

As a consequence, we obtain the following:

\begin{lem} Let $\z \in H^n (G, k)$ be a nonzero cohomology class of degree $n$ where $n \geq 1$ and let $\mu$ be the assignment as in Lemma \ref{lem:obstruction mu}. Then,  $\mu (v) = \la \z , v , \z \ra $ mod $(\z)$ for every $v \in H^{i+1} (G,k)$, $i\geq 0$, which satisfies $\z v=0$.
\end{lem}

\begin{proof} If we take $u=w=\z$ and $s=i+1$ in the second commuting diagram above, we obtain a diagram similar to the first commuting diagram above. We just need to show that $[\bP(\z), \Sig ^{n+i} \bP]$ is isomorphic to $\Ext ^i _{kG} (\Lz, k)$. Note that for $i\geq 0$, we have $$[\bP(\z), \Sig ^{n+i} \bP] \cong [\Gamma _n \bP (\z) , \Sig ^{n+i} \bP]$$ where $\Gamma _n \bPz $ denotes the truncation of $\bP$ at $n$.  The complex $\Gamma _n \bPz$ has trivial cohomology
except at dimension $n$ and $H_n (\G _n \bPz , k)\cong \Lz$. So, $\Gamma _n \bPz$ is homotopy equivalent to $\Sig ^n \bP (\Lz)$ where $\bP (\Lz )$ is a projective resolution of $\Lz$. This gives that $$[\bP(\z), \Sig ^{n+i} \bP] \cong [\Sig ^n \bP(\Lz), \Sig ^{n+i} \bP]\cong \Ext ^i _{kG} (\Lz , k).$$ This completes the proof.
\end{proof}

Combining the lemmas above, we obtain the following:

\begin{prop}\label{pro:main3} Let $\z \in H^n (G, k)$ be a nonzero cohomology class of degree $n$ where $n \geq 1$. Then, $\z$ is semi-productive if and only if $\la \z, v, \z \ra \equiv 0$ mod $(\z)$ for every $v \in H^* (G,k)$ which satisfies $\z v=0$.
\end{prop}

Proposition \ref{pro:main3} completes the proof of
$(i)\Leftrightarrow (ii)$ in Theorem \ref{thm:main3}. For the equivalence of the statements $(ii)$ and $(iii)$, we quote the following result by Hirsch \cite{Paper:Hirsch55Quelques}.

\begin{thm}[Hirsch \cite{Paper:Hirsch55Quelques}] Let $X$ be a simplicial complex. Then for every $u, \z \in H^* (X,\bF _2)$, we have $$\la \z , u, \z \ra \equiv u Sq^{n-1} \z \ \
\mod (\z).$$
\end{thm}

This completes the proof of Theorem \ref{thm:main3}. A proof for Hirsch's theorem can be found in \cite{Paper:Hirsch55Quelques}. Although the theorem is for $k=\bF_2$, one can easily extend the argument so that it holds for any field of characteristic $2$. In fact, for the cohomology of groups, one can easily give a separate proof for Hirsch's theorem using the methods in this paper.

\begin{rem} The similarities between the statements of Theorem \ref{thm:main2} and Theorem \ref{thm:main3} suggest that the Massey product approach could be used to prove Theorem \ref{thm:main2} as well. For this, one would need to generalize the notion of triple Massey product $\la \zeta, u, \zeta \ra$ to the case where $u$ is an element in $H^*(G,M)$ for some $kG$-module $M$. One would also need
to prove a more general version of Hirsch's theorem. We did not take this approach here since proving this more general version of Hirsch's theorem is not much shorter than proving Theorem \ref{thm:main2}. We also find the methods used in the proof of Theorem \ref{thm:main2} more interesting and possibly more useful for proving other theorems.
\end{rem}

We end the paper with an example which shows that being
semi-productive is a strictly weaker condition than being
productive.

\begin{ex}\label{ex:semi but not prod} Let $G=\bbZ/2 \times \bbZ/2$ and let $\{ x,y \}$ be a basis for $H^1(G,\bbF _2)$. Consider the class  $\z =x^2+xy+y^2 \in H^2 (G, \bbF _2)$. Since $\z$ is a nonzero divisor, it is semi-productive by Theorem \ref{pro:main3}. But, by Theorem \ref{thm:main1}, $\z$ is not productive since $Sq^1 (\z)=xy(x+y)$ is not divisible by $x^2+xy+y^2$.
\end{ex}

Unfortunately it is not as easy to find cohomology classes which are not semi-productive. For $k=\bbF_2$, we do not know if there exists a nonzero cohomology class $\z \in H^n(G, k)$ which is not semi-productive. On the otherhand, for an arbitrary field $k$ of characteristic $2$, it is possible to construct such examples. The following example is provided to us by Martin Langer. It comes from his earlier work on secondary multiplications in Tate cohomology (see \cite[Remark 3.6]{Paper:Langer2009SecondaryProducts}).
 
\begin{ex}\label{ex:MLangerexample} Let $G=Q_8$ be the quaternion group of order $8$ and $k$ be a field of characteristic $2$ which includes a primitive third root of unity. Let $\alpha \in k$ such that $\al ^2+\al+1=0$ and let $x,y$ be generators of $H^1(G, \bbF)$. Then, if we take $\z=\al x +y$ and $u=\al ^2 x+y$, then we get $\z u=0$, but $\widetilde{Sq}^0 (\z)u=(\al ^2 x+y)^2=\al x^2+y^2$ is not a multiple of $\z$ in $H^*(G,k)$. So, $\z$ is not semi-productive. 
\end{ex}

\providecommand{\bysame}{\leavevmode\hbox
to3em{\hrulefill}\thinspace}
\providecommand{\MR}{\relax\ifhmode\unskip\space\fi MR }
\providecommand{\MRhref}[2]{%
  \href{http://www.ams.org/mathscinet-getitem?mr=#1}{#2}
} \providecommand{\href}[2]{#2}

\end{document}